\documentclass[10pt,a4paper]{article}
\usepackage[latin1]{inputenc}
\usepackage{authblk}
\usepackage{amsmath}
\usepackage{amsfonts}
\usepackage{amssymb}
\usepackage{mathrsfs}
\usepackage{makeidx}
\usepackage{graphicx}
\usepackage{amsthm}
\usepackage{amscd}
\usepackage{bbm}
\usepackage[width=15.00cm, height=22.00cm]{geometry}
\newtheorem{theorem}{Theorem}[section]
\newtheorem{lemma}[theorem]{Lemma}

\newtheorem{corollary}[theorem]{Corollary}
\newtheorem*{theorem*}{Theorem}

\newtheorem*{conjecture}{Conjecture}
\usepackage[pdftex,bookmarks=true]{hyperref}
\usepackage{graphicx}
\usepackage{enumerate}
\usepackage{mathtools}
\mathtoolsset{showonlyrefs}

\theoremstyle{definition}
\newtheorem{definition}[theorem]{Definition} 
 
\newtheorem{example}[theorem]{Example} 
\newtheorem{eg}[theorem]{Example}
\newtheorem{remark}[theorem]{Remark}

\newcommand{\ee}[1]{\mathbb{E}\left[ #1 \right]} 
\newcommand{\cee}[2]{\mathbb{E}\left[\left. #1 \right| #2 \right]} 
\newcommand{\var}[1]{Var\left( #1 \right)}

\author{ Tulasi Ram Reddy \footnote{ Research supported in
		part by ISF-UGC post-doctoral fellowship. The results of this article are based on the author's PhD Thesis \cite{atrthesis} written at Department of Mathematics, IISc, Bangalore (supported in part by UGC, under SAP-DSA Phase IV).}  }
\AtEndDocument{\bigskip{\footnotesize%
		\textsc{Tulasi Ram Reddy, Theoretical Statistics and Mathematics Unit,  Indian Statistical Institute, Bangalore center.} \par  
		\textit{E-mail address}: \texttt{tulasi\_vs@isibang.ac.in} 
	}}

\title{Limiting empirical distribution of zeros and critical points of random polynomials agree in general.}

\date{}
\begin{document}
\maketitle

\begin{abstract}
	In this article, we study critical points (zeros of derivative) of random polynomials. Take two deterministic sequences $\{a_n\}_{n\geq1}$ and $\{b_n\}_{n\geq1}$ of complex numbers whose limiting empirical measures are same. By choosing $\xi_n = a_n$ or $b_n$ with equal probability, define the sequence of polynomials by $P_n(z)=(z-\xi_1)\dots(z-\xi_n)$. We show that the limiting measure of zeros and critical points agree for this sequence of random polynomials under some assumption. We also prove a similar result for triangular array of numbers. A similar result for zeros of generalized derivative (can be thought as random rational function) is also proved. Pemantle and Rivin initiated the study of critical points of random polynomials. Kabluchko proved the result considering the zeros to be i.i.d. random variables.
\end{abstract}

\section{Introduction}
The oldest known result relating the  zeros and critical points of a polynomial is Gauss-Lucas theorem, which states that the critical points of any polynomial with complex coefficients lie inside the convex hull formed by the zeros of the polynomial. In general nothing more can be said. Our interest in this article is dealing with sequences of polynomials, usually randomness included, with increasing degrees. We consider the case in which the point cloud made from the zeros of these polynomials converges to a probability measure in the complex plane. We want to understand the behavior of critical points of these sequences of polynomials. We recall the definition of weak convergence.

\begin{definition}
	For a sequence of probability measures, $\{\mu_n\} \text{ and } \mu$ on $\text{$\mathbb{C}$}$, we say that $\mu_n \xrightarrow{w} \mu$ \textit{weakly}, if  for any $ f\in C_c^{\infty}(\text{$\mathbb{C}$})$, we have $\lim\limits_{n \rightarrow \infty}\int_{X}^{}fd\mu_n = \int_{X}^{}fd\mu$. 
\end{definition}

We deal with sequence of complex numbers whose empirical measure converge to a probability measure. 

\begin{definition}
	Let $\{a_{n}\}_{ n \geq 1}$ $\left(\{a_{n,i}\}_{n\geq1; 1 \leq i \leq n}\right)$ be a sequence (triangular array) of numbers such that $\frac{1}{n}\sum_{i=1}^{n}\delta_{a_{i}}$ $\left(\frac{1}{n}\sum_{i=1}^{n}\delta_{a_{n,i}}\right)$ converge weakly to a probability measure $\mu$, we call such a sequence (triangular array) to be $\mu$-distributed. Here $\delta_a$ denotes the Dirac measure supported at $a$. 
\end{definition}

The question was raised by Pemantle and Rivin in \cite{pemantle} whether it is true that if the limiting measure of zeros of a sequence of polynomials converge to a probability measure $\mu$, then the corresponding limiting measure of critical points also converge to $\mu$. It is false in general as there are counterexamples. For convenience we will introduce the following notation. For any polynomial $P$, let $Z(P)$ denote the multi-set of zeros of $P$ and $\text{$\mathscr{M}$}(P)$ to be the uniform probability measure on $Z(P)$.   

In the case where all the zeros of the polynomials are real, because the zeros and critical points interlace, the empirical measure of zeros and critical points agree in limit. We now look at some examples where the limiting measure of zeros and critical points do not agree.

 The most commonly quoted~\cite{pemantle} sequence of polynomials where the limiting measure of zeros and limiting measure of critical points do not agree is $P_n(z)=z^n-1$. In this case the limiting zero measure $\lim\limits_{n\rightarrow \infty}\mathscr{M}(P_n)$ is the uniform probability measure on $S^1$ and the limiting critical point measure $\lim\limits_{n\rightarrow \infty}\mathscr{M}(P_n')$ is the Dirac measure at origin. In the spirit of this example we construct new set of examples for which the limiting measures of zeros and critical points are different.

\begin{eg}
	Recall that if a polynomial has all zeros real, then all its critical points have to be real and are interlaced between the zeros of the polynomial. Consider the polynomial $P_n(z)=(z-a_1^n)(z-a_2^n)\dots(z-a_k^n)$, where $a_1,a_2,\dots,a_k$ are real numbers such that $0 < a_1 < a_2 < \dots < a_k$. Define the sequence of polynomials to be $Q_n(z)=P_n(z^n)$, then $Q_n'(z)=nz^{n-1}P_n'(z^n)$. The zero set of $Q_n$ is \[Z(Q_n)=\bigcup\limits_{j=1}^{k}\bigcup\limits_{\ell=1}^{n}\{a_je^{ 2\pi i\frac{\ell}{n}}\}.\] Where as the zero set of $Q_n'$ is \[Z(Q_n')=\left(\bigcup\limits_{j=1}^{k-1}\bigcup\limits_{\ell=1}^{n} \{b_{j,n}^{\frac{1}{n}}e^{ 2\pi i\frac{\ell}{n}}\}\right)\bigcup\{0,0,\dots,0\},\] where $b_{1,n},b_{2,n},\dots,b_{k-1,n}$ are the zeros of the polynomial $P_n'(z)$. The probability measure $\text{$\mathscr{M}$}(Q_n')$ has mass $\frac{n-1}{kn-1}$ at $0$, hence its limiting measure will have mass $\frac{1}{k}$ at $0$. On the other hand the probability measure $\text{$\mathscr{M}$}(Q_n)$ is supported on $\bigcup\limits_{j=1}^{k}a_jS^1$. Hence the limiting measures do not agree.
\end{eg}
\begin{eg}
 Choose a polynomial $P$, whose zeros are in $\text{$\mathbb{D}$}_r$, where $r<1$. Define $Q_n(z)=P^n(z)-1$, then $Q_n'=nP^{n-1}(z)P'(z)$. If $z$ is a zero of $Q_n(z)$, then it satisfies $P^n(z)=1$, or $|P(z)|=1$. Therefore the limiting zero measure of $Q_n(z)$ is supported on the boundary of the polynomial lemniscate $\{z:|P(z)|\leq1\}$ of the polynomial $P$. The limiting zero measure for the sequence $\{Q_n\}_{n \geq 1}$ exists because $Q_n$ is the $nk$-th Chebyshev polynomial of the polynomial lemniscate of $P$. Hence the limiting zero measure is the equilibrium measure for the domain $\{z:|P(z)|\leq1\}$ (see Chapter 5 in \cite{ransford}). Where as, if $z_1,z_2,\dots,z_k$ are the roots of the polynomial $P$, then the limiting zero distribution of $Q_n'$ will be $\frac{1}{k}\sum_{i=1}^{k}\delta_{z_i}$. Hence the limiting measures of zeros and critical points of the given sequence of polynomials do not agree. 
\end{eg}

Before we discuss the above question, we recall the modes of convergence for random measures.
\begin{definition}\label{modes of convergence}
	Let $\mathcal{P}(\mathbb{C})$ be the set of probability measures on the complex plane, equipped with \textit{weak topology}. Let $\{\mu_n\}_{n\geq1}$ be a sequence in $\mathcal{P}(\mathbb{C})$ and   $\mu \in \mathcal{P}(\mathbb{C})$ we say,
	\begin{itemize}
		\item $\mu_n \xrightarrow{w} \mu$ in probability if $\lim\limits_{n\rightarrow\infty}\mathbb{P}(\mu_n \in N_\mu)=1$ for any  neighbourhood $N_\mu$ of $\mu$,
		
		\item $\mu_n \xrightarrow{w} \mu$ almost surely if $\mathbb{P}(\lim\limits_{n\rightarrow\infty}\mu_n \in N_\mu)=1$ for any neighbourhood $N_\mu$ of $\mu$.
	\end{itemize}
\end{definition}

Pemantle and Rivin in \cite{pemantle} considered a sequence of random polynomials whose zeros are i.i.d. with law $\mu$ having finite 1-energy and proved that the empirical law of critical points converge weakly to the same probability measure $\mu$.

We prove the result for a specific class of sequences (triangular arrays) which we call as \textit{log-Ces\'{a}ro-bounded} which is defined as follows.

\begin{definition}\label{def:log-cesaro}
	We say a sequence (triangular array) of complex numbers $\{a_{n}\}_{n\geq1}$ $\left(\{a_{n,i}\}_{n\geq1; 1 \leq i \leq n}\right)$ to be \textit{log-Ces\'{a}ro-bounded} if the Ces\'{a}ro means of the positive part of their logarithms are bounded i.e., the sequence $\{\frac{1}{n}\sum_{i=1}^{n}\log_+|a_{i}|\}_{n\geq 1}$ $\left(\{\frac{1}{n}\sum_{i=1}^{n}\log_+|a_{n,i}|\}_{n\geq 1}\right)$ is bounded.
\end{definition}

\begin{theorem}\label{thm2}
	Let $\{a_k\}_{k\geq1}$ and $\{b_k\}_{k\geq1}$ be two $\mu$-distributed and log-Ces\'{a}ro bounded sequences of complex numbers. Additionally assume that, $a_k \neq b_k$ for infinitely many $k$.	Define the sequence of independent random variables $\xi_k$ such that $\xi_k = a_k $ or $b_k$ with equal probability, for $k\geq1$. Define the polynomials $P_n(z):=(z-\xi_1)(z-\xi_2)\dots(z-\xi_n)$. Then, $\text{$\mathscr{M}$}(P_n)\xrightarrow{w}\mu$ almost surely and $\text{$\mathscr{M}$}(P_n')\xrightarrow{w}\mu$ in probability.  

\end{theorem}

\begin{remark}
	For the assertion of the above Theorem \ref{thm2} to hold, it is necessary to assume that the two sequences differ in infinitely many terms. Suppose not, we may choose one of the sequence to be a sequence for which the assertion of the theorem doesn't hold. Since both the sequences differ only in finitely many terms, the resulting sequence will be same as that of the sequence for which the assertion doesn't hold, with positive probability. Hence the statement of the Theorem \ref{thm2} doesn't hold.
\end{remark}
In the following example we will see a deterministic sequence, where the limiting empirical measures of zeros and critical points do not agree for the sequence of polynomials made through considering the terms of the sequences as zeros.
\begin{example}
Let the sequence $\{z_n\}_{n\geq1}$, be defined recursively as follows. $z_1=1$, $z_2=-1$ and for $1 \leq k\leq 2^{n+1}$, define $z_{2^n+k}=z_ke^{\frac{2\pi i}{2^{n+1}}}$. It can be verified that this sequence is $\mu$-distributed, where $\mu$ is uniform probability measure on $S^1$. Define $P_n(z)=\prod\limits_{k=1}^{n}(z-z_k)$. Then $\mathscr{M}(P_n)\rightarrow \mu$ whereas $\mathscr{M}(P_{2^n}')=\delta_0$.
\end{example}

Theorem \ref{thm2} can be used to obtain corollaries of the following form. Choose a deterministic sequence which is $\mu$-distributed and perturb each of its term by a random variable with diminishing variances. It can be shown that the empirical measure of the critical points of the polynomial, made from the perturbed sequence also converge to the same limiting probability measure $\mu$.

\begin{corollary}\label{Symmetric perturbations}
	Let $\{u_n\}_{n\geq1}$ be a $\mu$-distributed and log-Ces\'{a}ro bounded sequence. Let $\{v_n\}_{n\geq1}$ be the sequence such that $v_n=u_n+\sigma_nX_n$, where $X_n$s are i.i.d  random variables satisfying $X_n\stackrel{d}{=}-X_n$, $\ee{\text{$|X_n|$}}<\infty$  and $\sigma_n \downarrow 0$, $\sigma_n\neq0$. Define the polynomial $P_n(z):=(z-v_1)(z-v_2)\dots(z-v_n)$. Then,  $\text{$\mathscr{M}$}(P_n)\xrightarrow{w}\mu$ almost surely and  $\text{$\mathscr{M}$}(P_n')\xrightarrow{w}\mu$  in probability.
\end{corollary}

\begin{remark}
	In Corollary \ref{Symmetric perturbations}, we may choose the random variables $X_n$s to have complex Gaussian distribution or uniform distribution on the unit disk centered at $0$. In the case of complex Gaussian distributed random variables we get the result for unbounded perturbations and in the case of uniformly distributed random variables the perturbations are bounded. It can also be proved in the case where the limiting measure of zeros and the perturbed random zeros are supported on a same one dimensional set (for example unit circle).

\end{remark}

It is an easy fact (Page 15 in \cite{manjubook}) that if $\{X_n\}_{n \geq 1}$ is a sequence of i.i.d random variables that are not identically $0$ such that $\ee{\textbf{$\log_+|X_1|$}}<\infty$, then $\limsup\limits_{n \rightarrow \infty}|X_n|^\frac{1}{n}=1$. 
A special case of Theorem \ref{kabluchko} can be obtained as a corollary of the Theorem \ref{thm2}. The special case being the one in which the probability measure $\mu$ in consideration has bounded $\log_+$-moment.

\begin{corollary}\label{corollary4:kabluchko}
	Let $\mu$ be any probability measure on $\mathbb{C}$  satisfying $\int\limits_{\mathbb{C}}\log_+|z|d\mu(z)<\infty$. Let $X_1,X_2,\dots, X_n$ be i.i.d random variables distributed according to $\mu$. Define the polynomials  $P_n(z):=(z-X_1)(z-X_2)\dots(z-X_n)$. Then,  $\text{$\mathscr{M}$}(P_n)\xrightarrow{w}\mu$ almost surely and  $\text{$\mathscr{M}$}(P_n')\xrightarrow{w}\mu$  in probability.
\end{corollary}

We present a similar result of Theorem \ref{thm2} for triangular arrays of numbers.
\begin{theorem}\label{thm-tri}
	Let $\{a_{k,i}\}_{k\geq1;1\leq i \leq k}$ and $\{b_{k,i}\}_{k\geq1;1\leq i \leq k}$ be two $\mu$-distributed and log-Ces\'{a}ro bounded triangular arrays of complex numbers. Additionally assume that, $\sum\limits_{i =1}^{n}\log_+\frac{1}{|a_{n,i} - b_{n,i}|}=o(n^2)$.	Define the sequence of independent random variables $\xi_{k,i}$ such that $\xi_{k,i} = a_{k,i} $ or $b_{k,i}$ with equal probability, for  $1 \leq i \leq k$ and $k\geq1$. Define the sequence of polynomials whose $n$-th term is given by  $P_n(z):=(z-\xi_{n,1})(z-\xi_{n,2})\dots(z-\xi_{n,n})$. Then, $\text{$\mathscr{M}$}(P_n)\xrightarrow{w}\mu$ almost surely and $\text{$\mathscr{M}$}(P_n')\xrightarrow{w}\mu$ in probability.  
\end{theorem}
\begin{example}\label{reeds}
In \cite{reeds}, the author studies the real zeros of the Cauchy location likelihood equation to estimate the location parameter of Cauchy random variables. Let $X_1, X_2, \dots, $ be i.i.d. Cauchy distributed random variables with the density $\frac{1}{\pi(1+x^2)}$. Notice that the zeros of the Cauchy location likelihood equation 
\begin{equation}\label{cauchyLE}
\sum\limits_{k=1}^{n}\frac{\partial}{\partial \theta}\log\frac{1}{\pi(1+(X_k-\theta)^2)}=0 ,
\end{equation}
are the critical points of the polynomial $P_n(z)=\prod_{k=1}^{n}(z-X_k+i)(z-X_k-i)$. Following the proof of the Theorem \ref{thm2} (tweaking Lemma \ref{kolmogorov-rogozin} appropriately), it can be shown that the limiting empirical measures of zeros and critical points agree for the sequence of the polynomials $\{P_n\}_{n\geq1}$. The limiting empirical measure of zeros $\mathscr{M}(P_n)$ is uniform mixture of Cauchy distribution supported on the lines $\Im(z)=\pm i$. As a consequence we get that the number of real zeros of Cauchy location likelihood equation is $\small o(n)$.
\end{example}

We return to the example of the sequence of polynomials whose $n$-th term is $P_n(z)=z^n-1$. By removing a zero from $P_n$, we can see that the empirical measures of zeros and critical points agree in limit. Define the sequence $\{Q_n\}_{n\geq1}$, where $Q_n(z)=\frac{P_{n+1}(z)}{z-1}.$ From the definition of $Q_n$, the limiting zero measure of the sequence $\{Q_n\}_{n\geq1}$ is the uniform probability measure on $S^1$. The derivative of these polynomials is \[Q_n'(z)=\frac{nz^{n+1}-(n+1)z^n+1}{(z-1)^2}.\] For any $|z|>1$, the polynomial $nz^{n+1}-(n+1)z^n+1$ does not vanish for any $n$ large enough as the term $nz^{n+1}$ dominates the rest of terms in absolute value. Similarly, for any $|z|<1$ for any $n$ sufficiently large enough, $1$ dominates the rest of the terms and hence the polynomials does not vanish . Therefore the limiting zero measure of the sequence $\{Q_n'\}_{n\geq1}$ is supported on $S^1$.


 To get the angular distribution of zeros of $Q_n'$ we use a bound of Erd\"{o}s-Tur\'{a}n for the discrepancy between a probability measure and uniform measure on $S^1$. We will sate the inequality in the case where the probability measure is the counting probability measure of zeros of a polynomial.

\begin{theorem}[Erd\"{o}s-Tur\'{a}n~\cite{erdos-turan}]
	Let $\{a_k\}_{0\leq k\leq N}$ be a sequence of complex numbers such that $a_0a_N\neq 0$ and let, $P(z)=\sum\limits_{k=0}^{N}a_kz^k.$ 
	Then, \[\bigg|\frac{1}{N}\nu_N(\theta,\phi)-\frac{\phi-\theta}{2\pi}\bigg|^2\leq\frac{C}{N}\log\bigg|\frac{\sum_{k=0}^{N}|a_k|}{\sqrt{|a_0a_N|}}\bigg|,\]
	for some constant $C$ and $\nu_N(\theta,\phi):=\#\{z_k:\theta\leq\arg(z_k)<\phi\}$, where $z_1,z_2,\dots,z_N$ are zeros of $P(z)$.
\end{theorem}

Applying the above inequality for the polynomial $(z-1)^2Q_n'(z)$,  we get the limiting zero measure of $Q_n'$ is uniform probability measure on $S^1$ which agrees with the limiting zero measure of $Q_n$. As an application of the forthcoming Theorem \ref{thm3}, we will see that if we choose random subsequence from a $\mu$-distributed sequence, then the limiting  distribution of zeros and critical points agree for the polynomials made from this random sequence.

The next result deals with counting the zeros and pole of a random rational function. The random rational function is defined as $L_n(z)=\sum_{k=1}^{n}\frac{a_k}{z-z_k}$. In a special case where $\sum_{k=1}^{n}a_k=n$ and $a_k>0$ for every $k=1,2,\dots,n$, it is called generalized Sz.-Nagy derivative. For a classical derivative all $a_k$s are equal to $1$. It is mentioned in \cite{rahmanbook} that the motivation in studying generalized derivative is that many of the results for classical derivatives extend to the generalized derivatives. In the case where the poles of the rational function are real and the weights are chosen from Dirichlet distribution the exatct density of zeros of the rational function can be exactly solved and is referred as Dixon-Anderson density (see Proposition 4.2.1 in \cite{forrester}).
\begin{theorem}\label{thm3}
	Let $a_1, a_2, \dots$ be i.i.d. random variables satisfying $\ee{\text{$|a_1|$}}\text{$< \infty$}$.  Let $\{z_n\}_{n\geq 1}$ be a sequence satisfying the condition that for Lebesgue a.e. $z \in \mathbb{C}$ there exists a compact set $K_z$ with $d(z,K_z)>0$ such that there are infinitely many $z_k$'s in $K_z$, and there is a point $\omega$ that is not a limit point of $z_n$'s. Define $L_n(z):=  \frac{a_1}{z-z_1}+\frac{a_2	}{z-z_2}\dots+\frac{a_n}{z-z_n} $. Then $\frac{1}{n}\Delta \log(|L_n(z)|)\rightarrow 0$ in probability, in the sense of distributions.
\end{theorem}
In the statement of the above theorem, there is a mention of the sequence $\{z_n\}_{n\geq1}$ satisfying that for Lebesgue a.e. $z \in \mathbb{C}$ there exists a compact set $K_z$ with $d(z,K_z)>0$ such that there are infinitely many $z_k$'s in $K_z$, and there is a point $\omega$ that is not a limit point of $z_n$'s. Several classes of sequences satisfy this condition. For example any bounded sequence or any sequence that is not dense and $\mu$-distributed for appropriate $\mu$ satisfies this condition. 

\begin{remark}
	In Theorem \ref{thm3}, let $L_n(z)=\frac{Q_n(z)}{P_n(z)}$. Where $Q_n(z)$ id defined to be the generalized derivative of the polynomial $P_n$. Then Theorem \ref{thm3} asserts that $\frac{1}{n}\Delta\log|L_n(z)|\rightarrow0$, which in turn imply that $\text{$\mathscr{M}$}(Q_n)-\text{$\mathscr{M}$}(P_n)\rightarrow 0$ in the sense of distributions. If we assume that the sequence $\{z_k\}_{k\geq1}$ is $\mu$-distributed then it follows that the limiting measure of critical points converge to $\mu$.
\end{remark}

As an application of Theorem \ref{thm3}  we choose a random subsequence of a $\mu$-distributed sequence and show that the limiting empirical measures of zeros and critical points agree. We state this result as the following corollary.

\begin{corollary}\label{corollary:thm3}
	Let $\{z_n\}_{n\geq1}$ be a $\mu$-distributed sequence that is not dense in $\mathbb{C}$, for a $\mu$ which is not supported on the whole complex plane. Choose a subsequence $\{z_{n_k}\}_{k \geq 1}$ at random that is, each of $z_n$ is part of subsequence with probability $p<1$ independent of others. Define the polynomials $P_k(z):=(z-z_{n_1})(z-z_{n_2})\dots(z-z_{n_k})$.  Then,  $\text{$\mathscr{M}$}(P_k)\xrightarrow{w}\mu$ almost surely and  $\text{$\mathscr{M}$}(P_k')\xrightarrow{w}\mu$  in probability.
\end{corollary}

We believe a strengthened version of above Corollary \ref{corollary:thm3} is true. We state it as the following conjecture.

\begin{conjecture}
Let $\{z_n\}_{n\geq1}$ be a $\mu$-distributed and log-C\'{e}saro bounded sequence. Define the sequence of random polynomials to be $P_n(z)=\frac{(z-z_1)(z-z_2)\dots(z-z_{n+1})}{z-z_{s_n}}$, where $s_n$ is a random number distributed uniformly on the set $\{1, 2, \dots, n+1\}$.  Then,  $\text{$\mathscr{M}$}(P_k)\xrightarrow{w}\mu$ almost surely and  $\text{$\mathscr{M}$}(P_k')\xrightarrow{w}\mu$  in probability.
\end{conjecture}

\section{Proofs of Corollaries \ref{Symmetric perturbations}, \ref{corollary4:kabluchko} and \ref{corollary:thm3}.} \label{sec3}
In Corollary \ref{Symmetric perturbations} we deal with perturbations of a $\mu$-distributed sequence. We expect that the perturbed sequence will also have the same limiting probability measure as of the original sequence. It is formally stated and proved in the following lemma.

\begin{lemma}\label{lemma:perturb_limit_measure}
	Let $\{a_n\}_{n\geq1}$ be a $\mu$-distributed sequence, $\sigma_n\downarrow0$ and $X_1,X_2,\dots$ are i.i.d. random variables. Then, $\{a_n+\sigma_nX_n\}_{n\geq1}$ is a $\mu$-distributed sequence almost surely.
\end{lemma}
\begin{proof}
	It is enough to show that for any $f\in C_c^\infty(\mathbb{C})$,
	\[\frac{1}{n}\sum\limits_{k=1}^{n}\left(f(a_k)-f(a_k+\sigma_kX_k)\right)\rightarrow 0,\]
	almost surely. Fix $\epsilon>0$,  choose $M$ such that $\mathbb{P}(|X_n|>M)<\epsilon.$ Then,
	\begin{align}
	\frac{1}{n}|\sum\limits_{k=0}^{n}(f(a_k)-f(a_k+\sigma_kX_k)| & \leq \frac{1}{n}\sum\limits_{k=1}^{n}|(f(a_k)-f(a_k+\sigma_kX_k))\text{$\mathbbm{1}$}\{|X_k|>M\}|\\&+\frac{1}{n}\sum\limits_{k=1}^{n}|(f(a_k)-f(a_k+\sigma_kX_k))\text{$\mathbbm{1}$}\{|X_k|\leq M\}|,\\
	&\leq\frac{ 2||f||_\infty}{n} \sum\limits_{k=1}^{n}\text{$\mathbbm{1}$}\{|X_k|>M\}+\frac{1}{n}\sum\limits_{k=1}^{n}|\sigma_kX_k|||f'||_\infty,\\
	&\leq \frac{ 2||f||_\infty}{n} \sum\limits_{k=1}^{n}\text{$\mathbbm{1}$}\{|X_k|>M\} + \frac{M||f'||_\infty}{n}\sum\limits_{k=1}^{n}\sigma_k. \label{eqn:perturb_limit_measure}
	\end{align}
	Using law of large numbers and $\sigma_n\downarrow0$in the above equation \ref{eqn:perturb_limit_measure} we have 
	\[
	\lim\limits_{n\rightarrow\infty}\frac{1}{n}\bigg|\sum\limits_{k=0}^{n}(f(a_k)-f(a_k+\sigma_kX_k)\bigg|  \leq 2||f||_\infty\epsilon.
	\]
	Because $\epsilon>0$ is arbitrary, we get  $\lim\limits_{n\rightarrow \infty}\frac{1}{n}\sum\limits_{k=1}^{n}\left(f(a_k)-f(a_k+\sigma_kX_k)\right)=0$.
\end{proof}

The main idea in proving the corollaries is that we condition the random sequences suitably, so that the resulting sequences satisfy the hypothesis of the Theorem \ref{thm2} and then apply to obtain the result. More formally, say we condition the sequence on the event $E$. Assume the conditioned sequence can be realized as a random sequence which satisfies the hypothesis of Theorem \ref{thm2}. Let $\nu_n^E$ be the empirical measure of the critical points of the degree-$n$ polynomial formed by conditioned sequence. Fix $\epsilon>0$, then 
\begin{align}
\mathbb{P}\left(d(\nu_n,\mu)>\epsilon\right) &=\ee{\text{$\text{$\mathbbm{1}$}\{d(\nu_n,\mu)>\epsilon\}$}},\\
&=\ee{\cee{\text{$\text{$\mathbbm{1}$}\{d(\nu_n,\mu)>\epsilon\}$}}{\text{$E$}}},\\
&=\ee{\text{$\text{$\mathbbm{1}$}\{d(\nu_n^E,\mu)>\epsilon\}$}}. \label{eqn:convergence}
\end{align}

We will use the following inequalities, whenever required.
\begin{align}
\log_+|ab| & \leq \log_+|a| + \log_+|b| \label{logplusprod}\\
\log_-|ab| & \leq \log_-|a| +\log_-|b| \label{logminusprod}\\
\log_+|a_1+a_2+ \dots + a_n| & \leq \log_+|a_1| + \log_+|a_2|+ \dots +\log_+|a_n| +\log(n)\label{logplussum}
\end{align}

\begin{remark}
	The inequality \eqref{logplussum} is obtained by using the inequalities
	$|a_1+\dots+a_n|\leq |a_1|+\dots+|a_n| \leq n\max\limits_{i\leq n}{|a_i|}$
	and 
	$\log_+(\max\limits_{i\leq n}|a_i|)\leq \log_+|a_1|+\dots+\log_+|a_n|.$
	
\end{remark}

\begin{lemma}\label{lemma:log-cesaro-bounded}
	Let $\{a_n\}_{n\geq1}$ be a sequence that is log-Ces\'{a}ro bounded and $\{b_n\}_{n\geq1}$ be a sequence such that $b_n=a_n+\sigma_nX_n$, $\sigma_n\downarrow0$ and $X_1,X_2,\dots$ are i.i.d. random variables with $\ee{\text{$\log_+|X_1|$}}<\infty$. Then the sequence $\{b_n\}_{n\geq1}$ is also log-Ces\'{a}ro bounded.
\end{lemma}
\begin{proof}
	\begin{align}
	\frac{1}{n}\sum\limits_{k=1}^{n}\log_+|b_k| & \leq \frac{1}{n}\sum\limits_{k=1}^{n}(\log_+(|a_k|+|a_k-b_k|)),\\
	&\leq \frac{1}{n}\sum\limits_{k=1}^{n}\log_+|a_k|+\frac{1}{n}\sum\limits_{k=1}^{n}\log_+|\sigma_kX_k|+\log(2)\\
	&\leq \frac{1}{n}\sum\limits_{k=1}^{n}\log_+|a_k|+\frac{1}{n}\sum\limits_{k=1}^{n}\log_+|\sigma_k|+\frac{1}{n}\sum\limits_{k=1}^{n}\log_+|X_k|+\log(2) \label{eqn:lemma:log-cesaro:1}
	\end{align}
	The sequence $\{\frac{1}{n}\sum_{k=1}^{n}\log_+|\sigma_k|\}_{n\geq1}$ goes to $0$, because $\lim\limits_{n\rightarrow0}\sigma_n=0$. Using  law of large numbers and the fact that $\ee{\text{$\log_+|X_1|$}}<\infty$, the sequence $\{\frac{1}{n}\sum_{k=1}^{n}\log_+|X_k|\}_{n\geq1}$ is bounded almost surely. Combining \eqref{eqn:lemma:log-cesaro:1} and the above facts we get that the sequence $\frac{1}{n}\sum\limits_{k=1}^{n}\log_+|b_k|$ is bounded. This completes the proof.
\end{proof}

\begin{proof}[Proof of Corollary \ref{Symmetric perturbations}]
	
	Fix $r_n$ and $\theta_n$ for $n \geq 1$. Choose  $E=\{w:X_n(w)=\pm r_ne^{i\theta_n} \linebreak\text{ for }n\geq1\}$. Because $X_n$s are symmetric random variables,  the $n^{th}$ term of the resulting sequence will be $u_n + \sigma_nr_ne^{i\theta_n}$ or $u_n - \sigma_nr_ne^{i\theta_n}$ with equal probability independent of other terms. Choose $a_n=u_n+\sigma_nr_ne^{i\theta_n}$ and $b_n=u_n-\sigma_nr_ne^{i\theta_n}$.	 We need to show that almost surely the sequences $\{a_n\}_{n\geq1}$ and $\{b_n\}_{n\geq1}$ satisfy the hypotheses of the Theorem \ref{thm2}. It follows from Lemmas \ref{lemma:perturb_limit_measure} and \ref{lemma:log-cesaro-bounded} the sequences $\{a_n\}_{n \geq 1}$ and $\{b_n\}_{n \geq 1}$ are $\mu$-distributed and log-Ces\'{a}ro bounded almost surely. 
\end{proof}

\begin{proof}[Proof of Corollary \ref{corollary4:kabluchko}]
	If $\mu$ is a degenerate probability measure then the result is trivial to verify. If $\mu$ is not deterministic then choose two independent sequences of random numbers $\{a_n\}_{n\geq1}$ and $\{b_n\}_{n\geq1}$, where $a_n$s and $b_n$s are i.i.d random numbers obtained from measure $\mu$. Choose $X_n= a_n$ or $b_n$ with equal probability independent of other terms, then $\{X_n\}_{n\geq1}$ is a sequence of i.i.d random variables distributed according to probability measure $\mu$. Using the hypothesis $\int\limits_{\mathbb{C}}\log_+|z|d\mu(z)<\infty$ and applying law of large numbers for the random variables $\{\log_+|X_n|\}_{n \geq 1}$, we get that the sequences $\{a_n\}_{n\geq1}$ and $\{b_n\}_{n\geq1}$ are log-Ces\'{a}ro bounded almost surely. Therefore the constructed sequences satisfy the hypothesis of the Theorem \ref{thm2}. 
\end{proof}
\begin{lemma}\label{random_seq_limit}
	Let $\{a_k\}_{k\geq1}$ and $\{b_k\}_{k\geq1}$ be two sequences which are $\mu$ and $\nu$ distributed respectively. Define a random sequence $\{\xi_k\}_{k\geq1}$, where $\xi_k=a_k$ with probability $p$ and $\xi_k=b_k$ with probability $1-p$. Then $\mu_n=\frac{1}{n}\sum\limits_{k=1}^{n}\delta_{\xi_k}$ weakly converge to $\lambda=p\mu+(1-p)\nu$ almost surely.
\end{lemma}
\begin{proof}
	It is enough to show that for any open set $U\subset\mathbb{C}$, $ \frac{1}{n}\sum\limits_{k=1}^{n}\mathbbm{1}\left\{\xi_k \in U\right\}$ converge to $\lambda(U) $ almost surely. But from a version of law of large numbers we know that if $X_1,X_2,\dots$ are independent random variables (not necessarily identical), then 
	\[
	\frac{1}{n}\sum\limits_{k=1}^{n}\left(X_k-\ee{\text{$X_k$}}\right)\xrightarrow{a.s}0
	\]
	provided that $\sum\limits_{k=1}^{\infty}\frac{1}{k^2}\var{X\textbf{$_k$}}<\infty$. Applying this to the random variables $\mathbbm{1}\left\{\xi_k \in U\right\}$ we get that $\frac{1}{n}\sum\limits_{k=1}^{n}\mathbbm{1}\left\{\xi_k \in U\right\}$ converge to $\lambda(U)$ almost surely.
\end{proof}

\begin{proof}[Proof of Corollary \ref{corollary:thm3}]
	
	Choose $a_1,a_2,\dots$ be i.i.d $\mbox{Bernoulli}(p)$ random variables. Let $\{k_n\}_{n\geq1}$  be a random sequence such that $a_{k_n}=1$ and $a_\ell=0$ whenever $\ell \notin \{k_1,k_2,\dots\}$. Define $L_n^{(1)}(z)=L_{k_n}(z)=\frac{P_n'(z)}{P_n(z)}$. It is enough to show that $\frac{1}{n}\Delta\log|L_{k_n}(z)|\rightarrow 0$ in probability. The sequences $\{a_n\}_{n\geq1}$ and $\{z_n\}_{n\geq1}$ satisfy the hypothesis of the Theorem \ref{thm3}. Therefore $\frac{1}{n}\Delta \log|L_n(z)|\rightarrow 0$ in probability. Because $\{L_{k_n}^{(1)}(z)\}_{n\geq1}$ is a subsequence of $\{L_n(z)\}_{n\geq1}$ it follows that $\frac{1}{k_n}\Delta \log|L_{k_n}^{(1)}(z)|\rightarrow 0$ in probability. Because $k_n$ is a negative binomial random variable with parameters $(n,p)$, we have $\frac{k_n}{n}\rightarrow p$ almost surely. Therefore, $\frac{1}{n}\Delta \log|L_{k_n}^{(1)}(z)|\rightarrow 0$ in probability.
\end{proof}
In the next section we provide proofs for both the Theorems \ref{thm2}, \ref{thm-tri} and \ref{thm3}.
\section{Proofs of Theorems \ref{thm2}, \ref{thm-tri} and  \ref{thm3}.}\label{sec4}
\label{ch:proofs5}
\subsection{Outline of proofs.}
The proofs here are adapted from the proof of  Kabluchko's theorem as presented in \cite{kabluchko}. The proofs involve in analyzing the function  $L_n(z):=\frac{P_n'(z)}{P_n(z)} = \sum\limits_{k=1}^{n}\frac{1}{z-\xi_k}$.
We shall prove the theorems by showing that the hypotheses of the Theorems  \ref{thm2}, \ref{thm-tri} and \ref{thm3} imply the following three statements. 

\begin{align}
&\text{For Lebesgue a.e. $z \in \mathbb{C}$ }  \text{ and for every }\epsilon>0, \lim\limits_{n\rightarrow \infty}\mathbb{P}\left(\frac{1}{n}\log|L_n(z)|>\epsilon\right)=0.  \tag{A1} \label{A1}\\
&\text{For Lebesgue a.e. $z \in \mathbb{C}$ } \text{ and for every }\epsilon>0, \lim\limits_{n\rightarrow \infty}\mathbb{P}\left(\frac{1}{n}\log|L_n(z)|<-\epsilon\right)=0. \tag{A2} \label{A2}\\
&\text{For any }r>0, \text{the sequence }\left\{\int_{\text{$\mathbb{D}$}_r}\frac{1}{n^2}\log^2|L_n(z)|\right\}_{n\geq 1} \text{ is tight.} \tag{A3} \label{A3}
\end{align}

Statements \eqref{A1} and \eqref{A2} assert that $\frac{1}{n}\log|L_n(z)|$ converge to $0$ in probability. Statement \eqref{A3} assert that the sequence $\{\int\limits_{\text{$\mathbb{D}$}_r}\frac{1}{n^2}\log^2|L_n(z)|\}_{n \geq 1}$ is tight. A lemma of Tao and Vu links the above two facts to yield that $\{\int\limits_{\text{$\mathbb{D}$}_r}\frac{1}{n}\log|L_n(z)|\}_{n \geq 1}$ converge to $0$ in probability. We state this lemma below.

\begin{lemma}[Lemma~3.1 in~\cite{taovu}]\label{lem:tao_vu}
	Let $(X,\mathcal{A},\nu)$ be a finite measure space and $f_n:X\to \mathbb{R}$, $n\geq 1$  random functions which are defined over a probability space $(\Omega, \mathcal{B}, \mathbb{P})$ and are jointly measurable with respect to $\mathcal{A}\otimes \mathcal{B}$.
	Assume that:
	\begin{enumerate}
		\item For $\nu$-a.e.\ $x\in X$ we have $f_n(x)\to 0$ in probability, as $n\to\infty$.
		\item For some $\delta>0$, the sequence $\int_X |f_n(x)|^{1+\delta} d\nu(x)$ is tight.
	\end{enumerate}
	Then, $\int_X f_n(x)d\nu(x)$ converge in probability to $0$.
\end{lemma}
Thus it follows from the above assertions \eqref{A1}, \eqref{A2}, \eqref{A3} and Lemma \ref{lem:tao_vu}, that  $\int\limits_{\text{$\mathbb{D}$}_r}\frac{1}{n}\log|L_n(z)|dm(z) \rightarrow 0$ in probability for any $r>0$. Choose any $f\in C_c^{\infty}(\mathbb{C})$, assume that $\mbox{support}(f) \subseteq \text{$\mathbb{D}$}_r$ and define $f_n(z)=\frac{1}{n}\left(\log|L_n(z)|\right)\Delta f(z)$. Because $f$ is a bounded function and $\frac{1}{n}\log|L_n(z)|$ satisfy the hypothesis of Lemma \ref{lem:tao_vu}, the functions $f_n$ also satisfy the hypothesis of Lemma \ref{lem:tao_vu}. Therefore we get that $\int\limits_{\text{$\mathbb{D}$}_r}f_n(z)dm(z)\rightarrow 0$ in probability.  Applying Green's theorem twice we have the identity,
\[\int_{\text{$\mathbb{D}$}_r}^{}f(z)\Delta\frac{1}{n}\log|L_n(z)| = \int_{\text{$\mathbb{D}$}_r}^{}\frac{1}{n}\log|L_n(z)|\Delta f(z)dm(z).
\] 
The left hand side of the above integral is defined in the sense of distributions. Therefore it follows that
$\int_{\text{$\mathbb{D}$}_r}^{} f(z)\frac{1}{n}\Delta\log|L_n(z)|\rightarrow 0 $ in probability. This suffices for Theorem \ref{thm3}. We complete the proof of Theorem \ref{thm2} by the following arguments. In the sense of distributions we have

\begin{equation}\label{eqn:distribution}
\int_{\text{$\mathbb{D}$}_r}^{} f(z)\frac{1}{n}\Delta\log|L_n(z)| = \frac{1}{n}\sum\limits_{k=1}^{n}f(\xi_k)-\frac{1}{n}\sum\limits_{k=1}^{n-1}f(\eta_k^{(n)})
\end{equation}

From Lemma \ref{random_seq_limit} it follows that the sequence $\{\xi_n\}_{n\geq1}$ is $\mu$-distributed. Hence \linebreak $\frac{1}{n}\sum_{k=1}^{n}f(\xi_k) \rightarrow \int\limits_{\text{$\mathbb{D}$}_r}f(z)d\mu(z)$ almost surely. Therefore from \eqref{eqn:distribution} we get, \begin{equation}
\frac{1}{n}\sum_{k=1}^{n-1}f(\eta_k^{(n)}) \rightarrow \int\limits_{\text{$\mathbb{D}$}_r}f(z)d\mu(z) \text{  in probability.} \label{proof_1}
\end{equation}  Because for any $f \in C_c^\infty(\text{$\mathbb{C}$})$ and $\epsilon>0$, the sets of the form $\{\mu:|\int\limits_{\text{$\mathbb{C}$}}f(z)d\mu(z)|<\epsilon\}$ form an open base at origin, from Definition \ref{modes of convergence} and \eqref{proof_1} it follows that  $\frac{1}{n-1}\sum_{k=1}^{n-1}\delta_{\eta_i^{(n)}} \xrightarrow{w} \mu$ in probability.

We show \eqref{A1}, by obtaining moment bounds for $L_n(z)$. To show \eqref{A2} we will use a concentration bound for the function $L_n(z)$. In either of the Theorems \ref{thm2} and \ref{thm3}, observe that $L_n(z)$ is a sum of independent random variables. We state a  version of Kolmogorov-Rogozin inequality below to be used later in the proofs to get the concentration bounds for $L_n(z)$.

\paragraph{Kolmogorov-Rogozin inequality (multi-dimensional version)}\label{KR-Inequality} [Corollary 1. of Theorem 6.1 in \cite{KR1}.]
Let $X_1,X_2, \dots $ be independent random vectors in $\mathbb{R}^\text{$n$}$. Define the concentration function, 
$
Q(X,\delta) := \sup_{a\in \mathbb{R}^\text{$n$}}\mathbb{P}(X \in B(a,\delta)).
$
Let $\delta_i \leq \delta$ for each $i$, then 
\begin{equation}
Q(X_1+\dots + X_n,\delta) \leq \frac{C\delta}{\sqrt{\sum_{i=1}^{n}\delta_i^2(1-Q(X_i,\delta_i))}}.\label{kol-rog-ineq}
\end{equation}

It remains to show that the hypotheses of Theorems \ref{thm2} and \ref{thm3} imply \eqref{A1}, \eqref{A2} and \eqref{A3}. We show this in the subsequent sections.

\subsection{Proofs of Theorems \ref{thm2} and \ref{thm-tri}}

In the following lemma we show that the hypothesis of the Theorem \ref{thm2} imply \eqref{A1}.

\begin{lemma}\label{momentbound}
	Let $\{s_{n,k}\}_{n\geq1;1\leq k\leq n}$ be any triangular array of numbers. Define $L_n(z)=\sum\limits_{k=1}^{n}\frac{1}{z-s_{n,k}}$. Then for any $\epsilon > 0$, 
	and for Lebesgue a.e. $z \in \mathbb{C}$, $$\limsup\limits_{n\rightarrow \infty}\frac{1}{n}\log|L_n(z)|<\epsilon.$$
\end{lemma}
\begin{proof}
	Define  $A_n^{\epsilon}=\bigcup\limits_{k=1}^{n}\{z:|z-s_{n,k}|<e^{-n\epsilon} \}$ and $F^{\epsilon}=\limsup\limits_{n \rightarrow \infty} A_n^{\epsilon}$, then $F^\epsilon$ are decreasing sets in $\epsilon$. For these sets we have $\sum\limits_{n=1}^{\infty}m(A_n^{\epsilon}) \leq \sum\limits_{n=1}^{\infty}2\pi ne^{-2n\epsilon} < \infty$, where $m$ is Lebesgue measure on complex plane. Applying Borel-Cantelli lemma to the sequence $\{A_n^{\epsilon}\}_{n\geq 1}$ we get $m(F^{\epsilon})=0$. Because $F^\epsilon$ are decreasing sets in $\epsilon$, we have that if  $F=\bigcup\limits_{\epsilon>0}F^{\epsilon}$, then $m(F)=0$. Choose $z \in F^c$, there is $N_z^{\epsilon}$ such that for any $n>N_z^{\epsilon}$ we have $z \notin A_n^{\epsilon}$. Therefore $\frac{1}{|z-\xi_n|}>e^{n\epsilon}$ is satisfied only for finitely many $n$. Hence we have $|L_n(z)|<M+ne^{n\epsilon}$, where $M$ is a finite number obtained from the sum of terms for which the inequality $\frac{1}{|z-\xi_n|}>e^{n\epsilon}$ is violated. It follows from here  $\limsup\limits_{n\rightarrow \infty}\frac{1}{n}\log|L_n(z)|<\epsilon$. Therefore for $z\notin F$, we have $\limsup\limits_{n\rightarrow \infty}\frac{1}{n}\log|L_n(z)|<\epsilon.$
\end{proof}

\begin{lemma}\label{kolmogorov-rogozin}
	Let $L_n(z)=\sum\limits_{k=1}^{n}\frac{1}{z-\xi_k}$ where $\xi_ks$ are as in the Theorem \ref{thm2}. Then for any $\epsilon > 0$, 
	and almost every $z$ we have $\lim\limits_{n\rightarrow \infty}\mathbb{P}(\frac{1}{n}\log|L_n(z)|\leq -\epsilon)= 0$.
\end{lemma}
\begin{proof}
	Fix $z \neq a_k \text{ or } b_k$ for any $k\geq 1$.
	From  Kolmogorov-Rogozin inequality \eqref{kol-rog-ineq} and taking $\delta_i=\delta=e^{-n\epsilon}$ we have, 
	
	\begin{equation}\label{kreq}
	\mathbb{P}\left(\bigg|\sum\limits_{k=1}^{n}\frac{1}{z-\xi_k}\bigg|<e^{-n\epsilon}\right) \leq \frac{C}{\sqrt{\sum_{k=1}^{n}(1-Q(\frac{1}{z-\xi_k},e^{-n\epsilon}))}}.
	\end{equation}
	We shall show that $\sum_{k=1}^{n}(1-Q(\frac{1}{z-\xi_k},e^{-n\epsilon}))$ goes to $\infty$. Observe that,
	\begin{align}
	Q\left(\frac{1}{z-\xi_k},e^{-n\epsilon}\right) =& \sup\limits_{\alpha \in \mathbb{C}}\mathbb{P}\left( \bigg|\frac{1}{z-\xi_k}-\alpha\bigg|<e^{-n\epsilon} \right)
	\leq \frac{1}{2},   
	\end{align}
	whenever $|\frac{1}{z-a_k}-\frac{1}{z-b_k}| > 2e^{-n\epsilon}$. 
		
	Define $S_n=\{k\leq n:|\frac{1}{z-a_k}-\frac{1}{z-b_k}| > 2e^{-n\epsilon}\}$. Notice that if $a_k \neq b_k$, then there is $N_k$ such that whenever $n>N_k$, we have $k\in S_n$. Because $a_k\neq b_k$ for infinitely many $k$, $|S_n|$ increases to infinity as $n \rightarrow \infty$. The denominator on the right hand side of \eqref{kreq} is at least $\sqrt{\frac{|S_n|}{2}}$.
	Therefore $\mathbb{P}\left(\bigg|\sum\limits_{k=1}^{n}\frac{1}{z-\xi_k}\bigg|<e^{-n\epsilon}\right)\leq \frac{C\sqrt{2}}{\sqrt{|S_n|}}\rightarrow 0$, as $n \rightarrow \infty$. 
\end{proof}

\begin{lemma}\label{kolmogorov-rogozin-tri}
	Let $L_n(z)=\sum\limits_{k=1}^{n}\frac{1}{z-\xi_{k,n}}$ where $\xi_{k,n}s$ are as in the Theorem \ref{thm-tri}. Then for any $\epsilon > 0$, 
	and almost every $z$ we have $\lim\limits_{n\rightarrow \infty}\mathbb{P}(\frac{1}{n}\log|L_n(z)|\leq -\epsilon)= 0$.
\end{lemma}
\begin{proof}
	Fix any $z \in \mathbb{C}$ that does not agree with any of the terms in the given arrays.
	From  Kolmogorov-Rogozin inequality \eqref{kol-rog-ineq} and taking $\delta_i=\delta=e^{-n\epsilon}$ we have, 
	
	\begin{eqnarray}\label{kreq-tri}
	\mathbb{P}\left(\bigg|\sum\limits_{k=1}^{n}\frac{1}{z-\xi_{k,n}}\bigg|<e^{-n\epsilon}\right) \leq \frac{C}{\sqrt{\sum_{k=1}^{n}(1-Q(\frac{1}{z-\xi_{k,n}},e^{-n\epsilon}))}}.
	\end{eqnarray}
	It is enough to show that $\sum_{k=1}^{n}(1-Q(\frac{1}{z-\xi_{k,n}},e^{-n\epsilon}))$ goes to $\infty$.
	
	 Because we have that
	$\sum\limits_{k=1}^{n}\log_+\frac{1}{|a_{k,n}-b_{k,n}|}=o(n^2)$, there are sets $C_n \in \{1,2,\dots,n\}$ such that $|C_n|=[\frac{3n}{4}]$ and such that whenever $k\in C_n$, we have $\log_+\frac{1}{|a_{k,n}-b_{k,n}|}=o(n)$. The given two triangular arrays are $\mu$-distributed. Therefore we can choose $M>0$ and $N\in \mathbb{N}$, such that for any $n>N$ we have $A_n=\{k:|a_{k,n}|>M\}$ and $B_n=\{k:|b_{k,n}|>M\}$ satisfying $|A_n|>\frac{3n}{4}$ and $|B_n|>\frac{3n}{4}$. For $k\in A_n\cap B_n\cap C_n$, we have 
\begin{eqnarray}\label{ineq-tri}
\big|\frac{1}{z-a_{k,n}}-\frac{1}{z-b_{k,n}}\big|&=&\frac{|a_{k,n}-b_{k,n}|}{|z-a_{k,n}||z-b_{k,n}|}\\ &\geq& \frac{|a_{k,n}-b_{k,n}|}{2|z|+|b_{k,n}|+|a_{k,n}|}\\&\geq& \frac{|a_{k,n}-b_{k,n}|}{2|z|+2M}.
\end{eqnarray}

	Let $\log_+\frac{1}{|a_{k,n}-b_{k,n}|}=\alpha_n=o(n)$. Therefore from \ref{ineq-tri} we have $\big|\frac{1}{z-a_{k,n}}-\frac{1}{z-b_{k,n}}\big| \geq \frac{e^{-\alpha_n}}{2(|z|+M)}$. Hence for sufficiently large $n$, and we get $Q(\frac{1}{z-\xi_{k,n}},e^{-n\epsilon})=\frac{1}{2}$. Because $|A_n\cap B_n\cap C_n|\geq \frac{n}{4}$, the sum  $\sum_{k=1}^{n}(1-Q(\frac{1}{z-\xi_{k,n}},e^{-n\epsilon}))$ is at least $\frac{n}{8}$. Therefore the right hand side of \ref{kreq-tri} approaches $0$ as $n\rightarrow \infty$.
\end{proof}

\begin{lemma}\label{tight}
		Let $\{s_{n,k}\}_{n\geq1;1\leq k\leq n}$ be any log-C\'{e}saro bounded triangular array of numbers. Define $L_n(z)=\sum\limits_{k=1}^{n}\frac{1}{z-s_{n,k}}$. Then, for any $r>0$, the sequence $\{\int_{\text{$\mathbb{D}$}_r}\frac{1}{n^2}\log^2|L_n(z)|dm(z)\}_{n\geq1}$ is bounded. 
\end{lemma}
\begin{proof}
	
	We will first decompose $\log|L_n(z)|$ into its positive and negative parts and analyze them separately. Let $\log|L_n(z)|=\log_+|L_n(z)|-\log_-|L_n(z)|$. Then,
	\[\int_{\text{$\mathbb{D}$}_r}\frac{1}{n^2}\log^2|L_n(z)|dm(z)= \int_{\text{$\mathbb{D}$}_r}\frac{1}{n^2}\log_+^2|L_n(z)|dm(z)+\int_{\text{$\mathbb{D}$}_r}\frac{1}{n^2}\log^2_-|L_n(z)|dm(z).\]
	Using \eqref{logplussum}, we get,
	\begin{align}
	\int_{\text{$\mathbb{D}$}_r}\frac{1}{n^2}\log_+^2|L_n(z)|dm(z) & = \int_{\text{$\mathbb{D}$}_r}\frac{1}{n^2}\log_+^2\bigg|\sum\limits_{k=1}^{n}\frac{1}{z-s_{n,k}}\bigg|dm(z),\\
	& \leq \int_{\text{$\mathbb{D}$}_r}\frac{1}{n^2}\left(\sum\limits_{k=1}^{n}\log_+\bigg|\frac{1}{z-s_{n,k}}\bigg|+\log(n)\right)^2dm(z).
	\end{align}
	
	Using the Cauchy-Schwarz inequality $(a_1+a_2+\dots+a_n)^2\leq n(a_1^2+a_2^2+\dots+a_n^2)$ for the above, we get,
	
	\begin{align}
	\int_{r\text{$\mathbb{D}$}}\frac{1}{n^2}\log_+^2|L_n(z)|dm(z) 
	& \leq \int_{\text{$\mathbb{D}$}_r}\frac{n+1}{n^2}\left(\sum\limits_{k=1}^{n}\log_+^2\bigg|\frac{1}{z-s_{n,k}}\bigg|+\log^2(n)\right)dm(z),\\
	& = \frac{n+1}{n^2}\sum\limits_{k=1}^{n}\int_{\text{$\mathbb{D}$}_r}\log_-^2|z-s_{n,k}|dm(z) + \frac{n+1}{n^2}\log^2(n)\pi r^2.\label{eqn:lemma:tight:1}
	\end{align}
	Because Lebesgue measure on complex plane is translation invariant, we have $$\int_{\text{$\mathbb{D}$}_r}\log_-^2|z-\xi|dm(z)=\int_{\xi+\text{$\mathbb{D}$}_r}\log_-^2|z|dm(z)\leq\int_{\text{$\mathbb{D}$}_1}\log_-^2|z|dm(z)<\infty.$$ Therefore $\sup\limits_{\xi\in \text{$\mathbb{C}$}}\int_{K}\log^2|z-\xi|dm(z)<\infty$ for any compact set $K \subset \mathbb{C}$ it can be seen that each of the terms in the final expression \eqref{eqn:lemma:tight:1} are bounded. Hence the sequence $\{\int_{\text{$\mathbb{D}$}_r}\frac{1}{n^2}\log_+^2|L_n(z)|dm(z)\}_{n\geq1}$ is bounded.
	
	We will now show that the sequence $\{\int_{\text{$\mathbb{D}$}_r}\frac{1}{n^2}\log_-^2|L_n(z)|dm(z)\}_{n\geq1}$ is bounded.
	Let $P_n(z)=\prod\limits_{k=1}^{n}(z-s_{n,k})$ and $P_n'(z)=n\prod\limits_{k=1}^{n-1}(z-\eta_{k}^{(n)})$. Applying inequality \eqref{logminusprod} and Cauchy-Schwarz inequality we get,
	
	\begin{align}
	\int_{\text{$\mathbb{D}$}_r}\frac{1}{n^2}\log_-^2|L_n(z)|dm(z) & = \int_{\text{$\mathbb{D}$}_r}\frac{1}{n^2}\log_-^2\bigg|\frac{P_n'(z)}{P_n(z)}\bigg|dm(z),\\
	& \leq  \int_{\text{$\mathbb{D}$}_r}\frac{2}{n^2}\log_-^2|P_n'(z)|dm(z)+ \int_{\text{$\mathbb{D}$}_r}\frac{2}{n^2}\log_-^2\bigg|\frac{1}{P_n(z)}\bigg|dm(z).
	\end{align}
	Again applying inequalities \eqref{logminusprod}, \eqref{logplusprod}, \eqref{logplussum} and Cauchy-Schwarz inequality to the above we obtain,
	\begin{align}
	\int_{\text{$\mathbb{D}$}_r}\frac{1}{n^2}\log_-^2 & |L_n(z)|dm(z)\\ 	
	& \leq \int_{\text{$\mathbb{D}$}_r}\frac{2}{n^2}\left(\sum\limits_{k=1}^{n-1}\log_-|z-\eta_{k}^{(n)}|\right)^2dm(z) + \int_{\text{$\mathbb{D}$}_r}\frac{2}{n^2}\left(\sum_{k=1}^{n}\log_+|z-s_{n,k}|\right)^2dm(z),\\
	& \leq \frac{2}{n}\sum\limits_{k=1}^{n-1}\int_{\text{$\mathbb{D}$}_r}\log_-^2|z-\eta_{k}^{(n)}|dm(z) \label{eqn:lemma:tight1}\\&+ 2\int_{\text{$\mathbb{D}$}_r}\left(\log(2)+\log_+|z|+\frac{1}{n}\sum\limits_{k=1}^{n}\log_+|s_{n,k}|\right)^2dm(z).
	\label{eqn:lemma:tight2}
	\end{align}

	From the hypothesis, we have that both the triangular array $\{s_{n,k}\}_{n\geq1;1\leq k\leq n}$ is log-Ces\'{a}ro bounded Therefore \eqref{eqn:lemma:tight2} is bounded uniformly in $n$. Using the fact that $\sup\limits_{\xi\in \text{$\mathbb{C}$}}\int_{K}\log^2|z-\xi|dm(z)<\infty$ , we get \eqref{eqn:lemma:tight1} is bounded uniformly in $n$.

	From the above facts we get that the sequence $\left\{\frac{1}{n^2}\int_{\text{$\mathbb{D}$}_r}\log^2|L_n(z)|dm(z)\right\}_{n\geq1}$ is bounded. 
\end{proof}
Lemmas \ref{momentbound}, \ref{kolmogorov-rogozin}, \ref{tight} show that the statements \eqref{A1}, \eqref{A2} and \eqref{A3} are satisfied. Hence the Theorem \ref{thm2} and Theorem \ref{thm-tri} are proved.

\subsection{Proof of Theorem \ref{thm3}}

We will prove the theorem when $\omega=0$ i.e, $0$ is not a limit point of the sequence $\{z_n\}_{n\geq1}$. For other cases we can translate all the points by $\omega$ and apply the theorem. We will first prove a general lemma for sequences of numbers which will later be used in proving the subsequent lemmas.
\begin{lemma}\label{liminf}
	Given any sequence $\{z_k\}_{k\geq1}$, where $z_k \in \mathbb{C}$, $\liminf\limits_{n \rightarrow \infty}\left(\inf\limits_{|z|=r}|z-z_n|^{\frac{1}{n}}\right) \geq 1$ for Lebesgue a.e. $r \in \mathbb{R^+}$ w.r.t Lebesgue measure. 
\end{lemma}
\begin{proof}
	Fix $\epsilon > 0$ and let $A_n = \{r>0:\inf\limits_{|z|=r}|z-z_n|< (1-\epsilon)^{n} \}$. Let $m$ denote the Lebesgue measure on the complex plane. Then,
	\begin{align}
	m\left(\left\{r>0 :\liminf\limits_{n \rightarrow \infty}(\inf\limits_{|z|=r}|z-z_n|^\frac{1}{n}) \leq (1-\epsilon)\right\}\right) & =  m\left(\limsup\limits_{n \rightarrow \infty}A_n\right) \\
	& \leq  \lim\limits_{k \rightarrow \infty}m\left(\mathop{\cup}_{n \geq k} A_n\right)  
	\end{align}
	If $r \in A_k, $ then from the definition of $A_k$ we have that $r \in [|z_k|-(1-\epsilon)^k,|z_k|+(1-\epsilon)^k] $. Hence we get,	
	\begin{align}
	m&\left(\left\{r>0 :\liminf\limits_{n \rightarrow \infty}(\inf\limits_{|z|=r}|z-z_n|^\frac{1}{n}) \leq (1-\epsilon)\right\}\right)\\ & \leq \lim\limits_{k\rightarrow \infty } \sum_{n=k}^{\infty}m\left(\left\{r:|z_n|-(1-\epsilon)^{n} \leq r \leq |z_n|+(1-\epsilon)^n\right\}\right)\\
	& \leq  \lim\limits_{k \rightarrow \infty } \sum_{n=k}^{\infty} 2(1-\epsilon)^n = 0
	\end{align}
	The above is true for every $\epsilon >0$, therefore $\liminf\limits_{n \rightarrow \infty}\left(\inf\limits_{|z|=r}|z-z_n|^{\frac{1}{n}}\right) \geq 1$ outside an exceptional set $E\subset\mathbb{R}^\text{$+$}$ whose Lebesgue measure is $0$.  
\end{proof}
Define the set $F=\{z:\liminf\limits_{n \rightarrow \infty}|z-z_n|^{\frac{1}{n}} < 1\}$. Because $0$ is not a limit point of $\{z_n\}_{n\geq1}$, we have $\liminf\limits_{n\rightarrow\infty}|z_n|^\frac{1}{n}\geq1$. Hence $0\notin F$. For $|z|=r$, we have \[\liminf\limits_{n \rightarrow \infty}|z-z_n|^\frac{1}{n}\geq\liminf\limits_{n \rightarrow \infty}\left(\inf\limits_{|z|=r}|z-z_n|^{\frac{1}{n}}\right).\]
Hence $F\subseteq \{z:|z|=r, r\in E\}$  and by invoking Fubini's theorem we get $m( \{z:|z|=r, r\in E\})=0$. From the above two observations it follows that $m(F)=0$.

The following lemma shows that the hypothesis of the Theorem \ref{thm3} implies \eqref{A1}.
\begin{lemma}\label{momentthm3}
	Let $L_n(z)$ be as in the Theorem \ref{thm3}. Then for any $\epsilon>0$, and Lebesgue a.e. $z \in \mathbb{C}$,
	$$\limsup\limits_{n \rightarrow \infty}\frac{1}{n}\log|L_n(z)|< \epsilon$$ almost surely.
\end{lemma}
\begin{proof}
	From the hypothesis, Lemma \ref{liminf} and Using Markov's inequality we get  $$\sum_{n=1}^{\infty}\mathbb{P}\left(\sup\limits_{|z|=r}\big|\frac{a_n}{z-z_n}\big|>e^{n\epsilon}\right) \leq \sum_{n=1}^{\infty}\sup\limits_{|z|=r}\frac{\ee{\text{$|a_n|$}}}{|z-z_n|e^{n\epsilon}}.$$ Denoting $t_n(r)=\sup\limits_{|z|=r}\bigg|\frac{1}{z-z_n}\bigg|$ we have
	\begin{equation}
	\sum_{n=1}^{\infty}\sup\limits_{|z|=r}\frac{\ee{\text{$|a_n|$}}}{|z-z_n|e^{n\epsilon}}=\sum_{n=1}^{\infty}\frac{\ee{\text{$|a_n|$}}}{e^{n\epsilon}}t_n(r). \label{eqn:power_series}
	\end{equation} 
	Because $a_n$s are i.i.d. random variables, $\ee{\text{$|a_n|$}}=\ee{\text{$|a_1|$}}$. Using the root test for the convergence of sequences and the Lemma \ref{liminf}, it follows that the right hand side of \eqref{eqn:power_series} is convergent for Lebesgue a.e. $r \in (0,\infty)$. Invoking Borel-Cantelli lemma we can say that $\sup\limits_{|z|=r} \frac{|a_n|}{|z-z_n|}>e^{n\epsilon}$ only for finitely many times. From here we get $|L_n(z)| \leq M_\epsilon + ne^{n\epsilon}$, where $M_\epsilon$ is a finite random number which is obtained by bounding the finite number of terms for which $\sup\limits_{|z|=r} \frac{|a_n|}{|z-z_n|}>e^{n\epsilon}$ is satisfied. Therefore we get that $\limsup\limits_{n \rightarrow \infty}\frac{1}{n}\log|L_n(z)|< \epsilon$ almost surely. 
\end{proof}
Notice that we have proved a stronger version of the Lemma \ref{momentthm3}. We will state this as a remark which will be used further lemmas.
\begin{remark}\label{one}
	Define $M_n(R):=\sup\limits_{|z|=R}|L_n(z)|$. Then for any $\epsilon>0$, we have $$\limsup\limits_{n \rightarrow \infty}\frac{1}{n}\log M_n(R)< \epsilon$$ for almost every $R>0$.
\end{remark}
For proving a similar result for the lower bound of $\log|L_n(z)|$ and establish \eqref{A2}, we need the Kolmogorov-Rogozin inequality  \ref{KR-Inequality} which was stated at the beginning of this chapter.

\begin{lemma}\label{krthm3}
	Let $L_n(z)$ be as in Theorem \ref{thm3}. Then for any $\epsilon>0$, and Lebesgue a.e. $z \in \mathbb{C}$ $$\lim\limits_{n \rightarrow \infty}\mathbb{P}\left(\frac{1}{n}\log|L_n(z)|<-\epsilon\right)=0.$$  
\end{lemma}

\begin{proof}
	Fix $z \in \mathbb{C}$ which is not in the exceptional set $F$. Let $z_{i_1},z_{i_2},\dots z_{i_{l_n}}$ be the points in $K_z$ from the set  $\{z_1,z_2,\dots,z_n\}$. From the definition of concentration function and the fact that the concentration function $ Q(X_1+X_2+\dots+X_n,\delta)$ is decreasing in $n$ we get,
	\begin{align}
	\mathbb{P}\left(|L_n(z)|\leq e^{-n\epsilon}\right) & \leq  Q\left(\sum_{i=1}^{n}\frac{a_i}{z-z_i},e^{-n\epsilon}\right), \\ 
	&\leq   Q\left(\sum_{k=1}^{l_n}\frac{a_{i_k}}{z-z_{i_k}},e^{-n\epsilon}\right). 
	\end{align} 
	The random variables $\frac{a_{i_k}}{z-z_{i_k}}$s are independent. Hence we can apply Kolmogorov-Rogozin inequality to get,
	\begin{align}
	\mathbb{P}\left(|L_n(z)|\leq e^{-n\epsilon}\right) & \leq  C_\epsilon\left\{\sum_{i=1}^{l_n}\left(1-Q\bigg(\frac{a_{i_k}}{z-z_{i_k}},e^{-n\epsilon}\bigg)\right)\right\}^{-\frac{1}{2}}.
	\end{align}
	Because $|z-z_{i_k}|\leq d(z,K_z)+diam(K_z)$, from above we get, 
	\begin{align}
	\mathbb{P}\left(|L_n(z)|\leq e^{-n\epsilon}\right) & \leq  C_\epsilon \left\{\sum_{i=1}^{l_n}\left(1-Q\left(a_{i_k},\left(d(z,K_z)+diam(K_z)\right)e^{-n\epsilon}\right)\right)  \right\}^{-\frac{1}{2}}\label{eqn:lemma5.4.8:1}
	\end{align}
	Because $a_{i_k}$s are non-degenerate i.i.d random variables and $l_n\rightarrow\infty$, the right hand side of \eqref{eqn:lemma5.4.8:1} converges to $0$ as $n\rightarrow\infty$. Hence the lemma is proved.
\end{proof}
It remains to show that the hypothesis of Theorem \ref{thm3} implies \eqref{A3}. Fix  $R>r$. The idea here is to write the function $\log|L_n(z)|$ for $z \in \text{$\mathbb{D}$}_r$ as an integral on the boundary of a larger disk $\text{$\mathbb{D}$}_R$ and bound the integral uniformly on the disk $\text{$\mathbb{D}$}_r.$ This is facilitated by Poisson-Jensen's formula for meromorphic functions.   
The Poisson-Jensen's formula is stated below. Let $\alpha_1, \alpha_2, \dots \alpha_k$ and $\beta_1,\beta_2, \dots \beta_\ell$ be the zeros and poles of a meromorphic function $f$ in $\text{$\mathbb{D}$}_R$. Then
\begin{align}	
\log|f(z)| = \frac{1}{2\pi}\int_{0}^{2\pi}\Re\bigg(\frac{Re^{i\theta}+z}{Re^{i\theta}-z}\bigg)\log|f(Re^{i\theta})|d\theta - \sum_{m=1}^{k}\log\bigg|\frac{R^{2}-\overline{\alpha}_jz}{R(z-\alpha_j)}\bigg|\\ + \sum_{m=1}^{l}\log\bigg|\frac{R^{2}-\overline{\beta}_jz}{R(z-\beta_j)}\bigg|
\end{align}

The following lemma \ref{tighttwo} gives an estimate of the boundary integral obtained in the Poisson-Jensen's formula when applied for the function $\log|L_n(z)|$ 	 at $z=0$. Define \[\text{$\mathcal{I}$}\text{$_n(z,R)$}:=\frac{1}{2\pi}\int_{0}^{2\pi}\Re\bigg(\dfrac{Re^{i\theta}+z}{Re^{i\theta}-z}\bigg)\log|L_n(Re^{i\theta})|d\theta.\] 
\begin{lemma}\label{tighttwo}
	There is a constant $c_2>0$ such that 
	\begin{align}
	\lim\limits_{n \rightarrow \infty}\mathbb{P}\left( \dfrac{1}{n}\text{$\mathcal{I}$}(0,R)\leq-c_2\right) = 0.
	\end{align}
\end{lemma}
\begin{proof}
	From Poisson-Jensen's formula at $0$ we get,
	\begin{equation}
	\dfrac{1}{n}\text{$\mathcal{I}$}\text{$_n(0;R) = \dfrac{1}{n}\log|L_n(0)| +\dfrac{1}{n}\sum_{m=1}^{k}\log\bigg|\dfrac{z_{i_m}}{R}\bigg| - \dfrac{1}{n}\sum_{m=1}^{l}\log\bigg|\dfrac{\alpha_{i_m}}{R}\bigg|$},\label{eqn:lemma5.4.9:1}
	\end{equation}
	
	where $z_{i_m}s$ and $\alpha_{i_m}s$ are zeros and critical points respectively of $P_n(z)$ in the disk $\text{$\mathbb{D}$}_R$. Because $0$ is not a limit point of $\{z_1,z_2, \dots\}$, $\left\{\dfrac{1}{n}\sum_{m=1}^{k}\log\big|\dfrac{z_{i_m}}{R}\big|\right\}_{n\geq1}$ is a sequence of negative numbers bounded from below. $\left\{\dfrac{1}{n}\sum_{m=1}^{l}\log\bigg|\dfrac{\alpha_{i_m}}{R}\bigg|\right\}_{n\geq1}$ is also a sequence of negative numbers. Therefore the last two terms in the right hand side of \eqref{eqn:lemma5.4.9:1} are bounded below. Because $0$ is not in exceptional set $F$, from Lemma \ref{krthm3} we have that the sequence $\lim\limits_{n\rightarrow\infty}\mathbb{P}\left(\frac{1}{n}\log|L_n(z)|<-1\right)=0$ is bounded from below. Therefore there exists $C_1$ such that $$\lim\limits_{n\rightarrow\infty}\mathbb{P}\left(\frac{1}{n}\log|L_n(z)|<-1 \text{ and }\dfrac{1}{n}\sum_{m=1}^{k}\log\bigg|\dfrac{z_{i_m}}{R}\bigg| - \dfrac{1}{n}\sum_{m=1}^{l}\log\bigg|\dfrac{\alpha_{i_m}}{R} <-C_1\right)=0.$$ Choosing $c_2=C_1+1$ the statement of lemma is established.
\end{proof}

Using above lemma \ref{tight} and exploiting formula of Poisson kernel for disk we will now obtain an uniform bound for the corresponding integral $\text{$\mathcal{I}$}_n(z,R)$.
\begin{lemma}\label{tight3}
	There is a constant $b>0$ such that for any $z \in \text{$\mathbb{D}$}\text{$_r$}$ 
	\begin{equation}
	\lim\limits_{n\rightarrow \infty}\mathbb{P}\left(\dfrac{1}{n}\text{$\mathcal{I}$}_n(z,R)\leq -b\right)=0.
	\end{equation}
\end{lemma}
\begin{proof}
	We will decompose the function $\log|L_n(z)|$ into its positive and negative components. Let $\log|L_n(z)|=\log_+|L_n(z)|-\log_-|L_n(z)|$, where $\log_+|L_n(z)|$ and $\log_-|L_n(z)|$ are positive. Using this we can write,
	
	\begin{align}
	2\pi\text{$\mathcal{I}$}\text{$_n(z)$} = & \int_{0}^{2\pi}\log|L_n(Re^{i\theta})|\Re\bigg(\dfrac{Re^{i\theta}+z}{Re^{i\theta}-z}\bigg)d\theta, \\
	= & \int_{0}^{2\pi}\log_+|L_n(Re^{i\theta})|\Re\bigg(\dfrac{Re^{i\theta}+z}{Re^{i\theta}-z}\bigg)d\theta -\int_{0}^{2\pi}\log_-|L_n(Re^{i\theta})|\Re\bigg(\dfrac{Re^{i\theta}+z}{Re^{i\theta}-z}\bigg)d\theta.
	\end{align}

	We can find constants $C_3$  and $C_4$ such that for any $z \in \text{$\mathbb{D}$}\text{$_r$}$,
	$0 < C_3 \leq \Re\bigg(\dfrac{Re^{i\theta}+z}{Re^{i\theta}-z}\bigg) \leq C_4 < \infty$ is satisfied. Therefore,
	\begin{align}
	2\pi\text{$\mathcal{I}$}\text{$_n(z)$} \geq & C_3\int_{0}^{2\pi}\log_+|L_n(Re^{i\theta})|d\theta - C_4\int_{0}^{2\pi}\log_-|L_n(Re^{i\theta})|^-d\theta,\\
	\geq & 2\pi C_3\text{$\mathcal{I}$}\textbf{$_n(0)-2\pi(C_4-C_3) M_n(R)$}.\label{eqn:lemma5.4.10:1}
	\end{align}
	From the Remark \ref{one} and Lemma \ref{tighttwo}  we get 
	\begin{equation}
	\lim\limits_{n \rightarrow \infty}\mathbb{P}\left( \dfrac{1}{n}\text{$\mathcal{I}$}\text{$_n(0)\leq -c \text{ or } \dfrac{1}{n}M_n(R) > 1 $}\right) = \textbf{$0$} \label{eqn:lemma5.4.10:2}
	\end{equation}
	
	The proof is completed from above \eqref{eqn:lemma5.4.10:2} and \eqref{eqn:lemma5.4.10:1} and by choosing $b=2\pi(cC_3+C_4-C_3)$.
\end{proof}

To complete the argument we now need to control the other terms in Poisson-Jensen's formula. Let $\xi_{i_m}s$ and $\beta_{i_m}s$ be the poles and zeros of $L_n(z)$ in $\text{$\mathbb{D}$}_R$ and  $k,l(\leq n)$ are the number of zeros and poles of $L_n(z)$ respectively in $\text{$\mathbb{D}$}_R$. Now applying Poisson-Jensen's formula to $L_n(z)$ we have,

\begin{align}	
\frac{1}{n^{2}}\int_{\text{$\mathbb{D}$}\text{$_r$}}^{}&\log^{2}|L_n(z)|dm(z)   \\ &=\frac{1}{n^{2}}\int_{\text{$\mathbb{D}$}\text{$_r$}}\left(\text{$\mathcal{I}$}\textbf{$_n(z)+\sum\limits_{m=1}^{k}\log\biggl|\frac{R(z-\beta_{i_m})}{R^2-\overline{\beta}_{i_m}z}\biggr|+\sum\limits_{m=1}^{l}\log\biggl|\frac{R(z-\xi_{i_m})}{R^2-\overline{\xi}_{i_m}z}\biggr|$}\right)^2dm(z) 
\end{align}
Invoking a case of Cauchy-Schwarz inequality  $(a_1+a_2+\dots+a_n)^2\leq n(a_1^2+a_2^2+\dots+a_n^2)$ repeatedly we get,

\begin{align}
\int_{\text{$\mathbb{D}$}\text{$_r$}}^{}\dfrac{1}{n^{2}}&\log^{2}|L_n(z)|dm(z)
\\ & \leq  \dfrac{3}{n^{2}}\int_{\text{$\mathbb{D}$}\text{$_r$}}^{}|\text{$\mathcal{I}$}_n(z)|^2dm(z) + \dfrac{3}{n^{2}}\int_{\text{ $\mathbb{D}$}_r}^{}\left(\sum_{m=1}^{k}\log\bigg|\dfrac{R(z-\beta_{i_m})}{R^{2}-\overline{\beta}_{i_m}z}\bigg|\right)^{2}dm(z)\\ 
&+\dfrac{3}{n^{2}}\int_{\mathbb{D}\text{$_r$}}^{}\left(\sum_{m=1}\log\bigg|\dfrac{R(z-\xi_{i_m})}{R^{2}-\overline{\xi}_{i_m}z}\bigg|\right)^{2}dm(z),\\
&\leq  \int_{\mathbb{D}\text{$_r$}}^{}\dfrac{3}{n^{2}}|\text{$\mathcal{I}$}_n(z)|^2dm(z) + \dfrac{3k}{n^{2}}\sum_{m=1}^{k}\int_{\mathbb{D}\text{$_r$}}^{}\log^{2}\bigg|\dfrac{R(z-\beta_{i_m})}{R^{2}-\overline{\beta}_{i_m}z}\bigg|dm(z)\\
&+\dfrac{3l}{n^{2}}\sum_{m=1}^{l}\int_{\mathbb{D}\text{$_r$}}^{}\log^{2}\bigg|\dfrac{R(z-\xi_{i_m})}{R^{2}-\overline{\xi}_{i_m}z}\bigg|dm(z).
\end{align}
For $z \in\text{ $\mathbb{D}$}_r$, we have $|R^2-\overline{\beta}_{i_m}z|\geq R(R-r)$. Applying this inequality in the above we get,
\begin{align}
\int_{\mathbb{D}\text{$_r$}}^{}\dfrac{1}{n^{2}}\log^{2}|L_n(z)|dm(z)
& \leq  \int_{\mathbb{D}\text{$_r$}}^{}\dfrac{3}{n^{2}}|\text{$\mathcal{I}$}_n(z)|^2dm(z) + \dfrac{3k}{n^{2}}\sum_{m=1}^{k}\int_{\mathbb{D}\text{$_r$}}^{}\log^{2}\bigg|\dfrac{z-\beta_{i_m}}{R-r}\bigg|dm(z)\\
&+\dfrac{3l}{n^{2}}\sum_{m=1}^{l}\int_{\mathbb{D}\text{$_r$}}^{}\log^{2}\bigg|\dfrac{z-\xi_{i_m}}{R-r}\bigg|dm(z).  \label{eqn:poisson-jensen:1}
\end{align}

From the Lemmas \ref{tighttwo} and \ref{tight3}, the corresponding sequence $\frac{3}{n^{2}}\int_{\mathbb{D}\text{$_r$}}^{}|\text{$\mathcal{I}$}$$_n(z)|^2dm(z)$ is tight. The function $\log^{2}|z|$ is an integrable function on any bounded set in $\mathbb{C}$. Combining these facts and above inequality \eqref{eqn:poisson-jensen:1} we have that the sequences \linebreak $\left\{\int_{\mathbb{D}\text{$_r$}}^{}\frac{1}{n^{2}}\log^{2}|L_n(z)|dm(z)\right\}_{n\geq1}$ are tight. Hence the hypothesis of the Theorem \ref{thm3} implies \eqref{A3}. Therefore the proof of the theorem is complete.

\section*{Acknowledgments} The author is grateful to M. Krishnapur for having several discussions during the course of this work. The author is also thankful to D. Zaporozhets for pointing the Example \ref{reeds}.
\bibliographystyle{plain}
\bibliography{paper}
\end{document}